\documentclass[12pt]{article}

\usepackage{amsmath,amssymb,amsthm,eucal,amscd,tikz,tikz-cd,mathtools}
\usepackage{xcolor}
\definecolor{allrefcolors}{rgb}{0,0.5,0.4}
\usepackage[pagebackref,linktocpage=true,colorlinks=true,allcolors=allrefcolors,bookmarksopen,bookmarksdepth=3]{hyperref}

\newtheorem*{theorem*}{Theorem}
\newtheorem{theorem}{Theorem}[section]

\newtheorem{lemma}[theorem]{Lemma}

\newtheorem{proposition}[theorem]{Proposition}

\theoremstyle{definition}
\newtheorem{definition}[theorem]{Definition}
\theoremstyle{remark}
\newtheorem{remark}[theorem]{Remark}
\newtheorem{question}[theorem]{Question}
\newtheorem{example}[theorem]{Example}

\numberwithin{equation}{section}

\newcommand{\LL}{\mathbb L}

\newcommand{\CC}{\mathbb C}
\newcommand{\RR}{\mathbb R}

\newcounter{rcount}

\newcounter{rcountsave}

\newcommand{\skull}{\mathfrak{s}}

\usepackage{epigraph}

\begin{document}

\title{Arboreal singularities from Lefschetz fibrations}

\author{Vivek Shende}

\date{}

\maketitle

\begin{abstract}
Nadler introduced certain Lagrangian singularities indexed by trees, and determined
their microlocal sheaves to be the category of modules over the corresponding tree quiver.  Another family of spaces
indexed by trees: the tree plumbings of spheres.  The Fukaya-Seidel category
of the Lefschetz fibration with this plumbing as fiber and all spheres as vanishing cycles is well
known to also be modules over the tree quiver.   Here we upgrade this matching of categories to 
a matching of geometry. 
\end{abstract}

\section{Introduction}

Fronts for the Legendrian links of what Nadler calls the $A_1, A_2, A_3$ {\em arboreal singularities}:

\begin{center}
\includegraphics[scale=0.25]{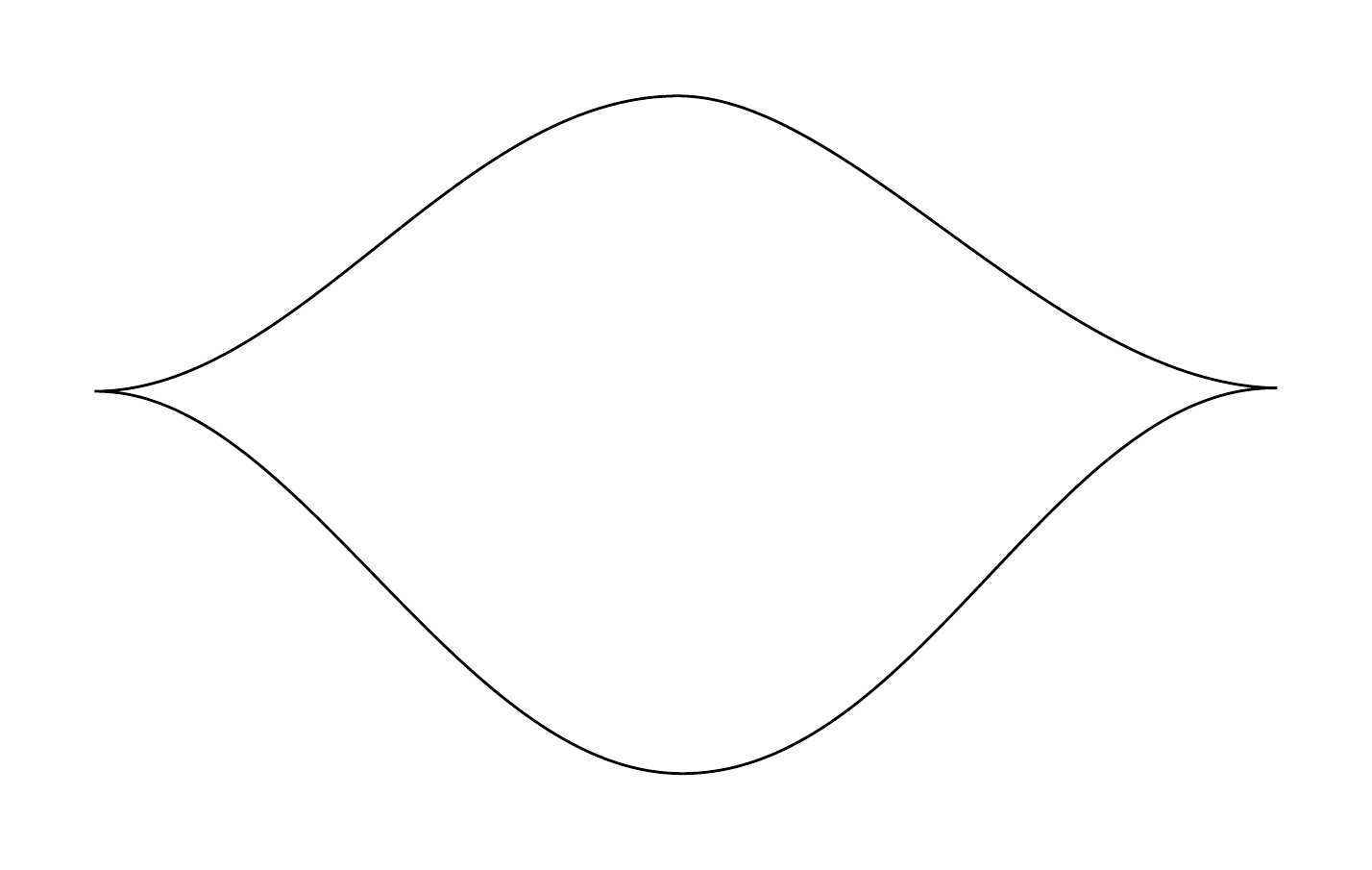} \includegraphics[scale=0.25]{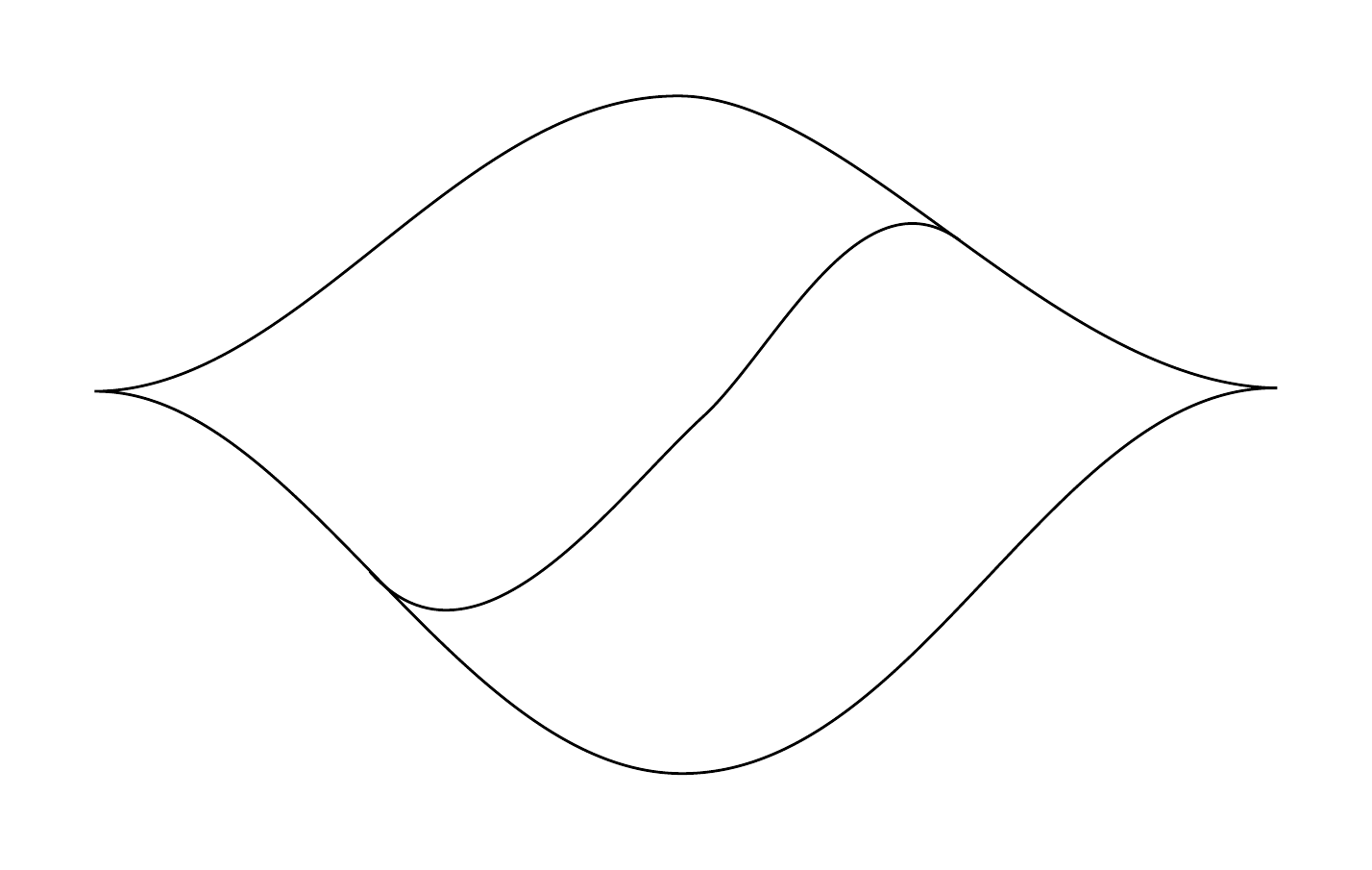} \includegraphics[scale=0.25]{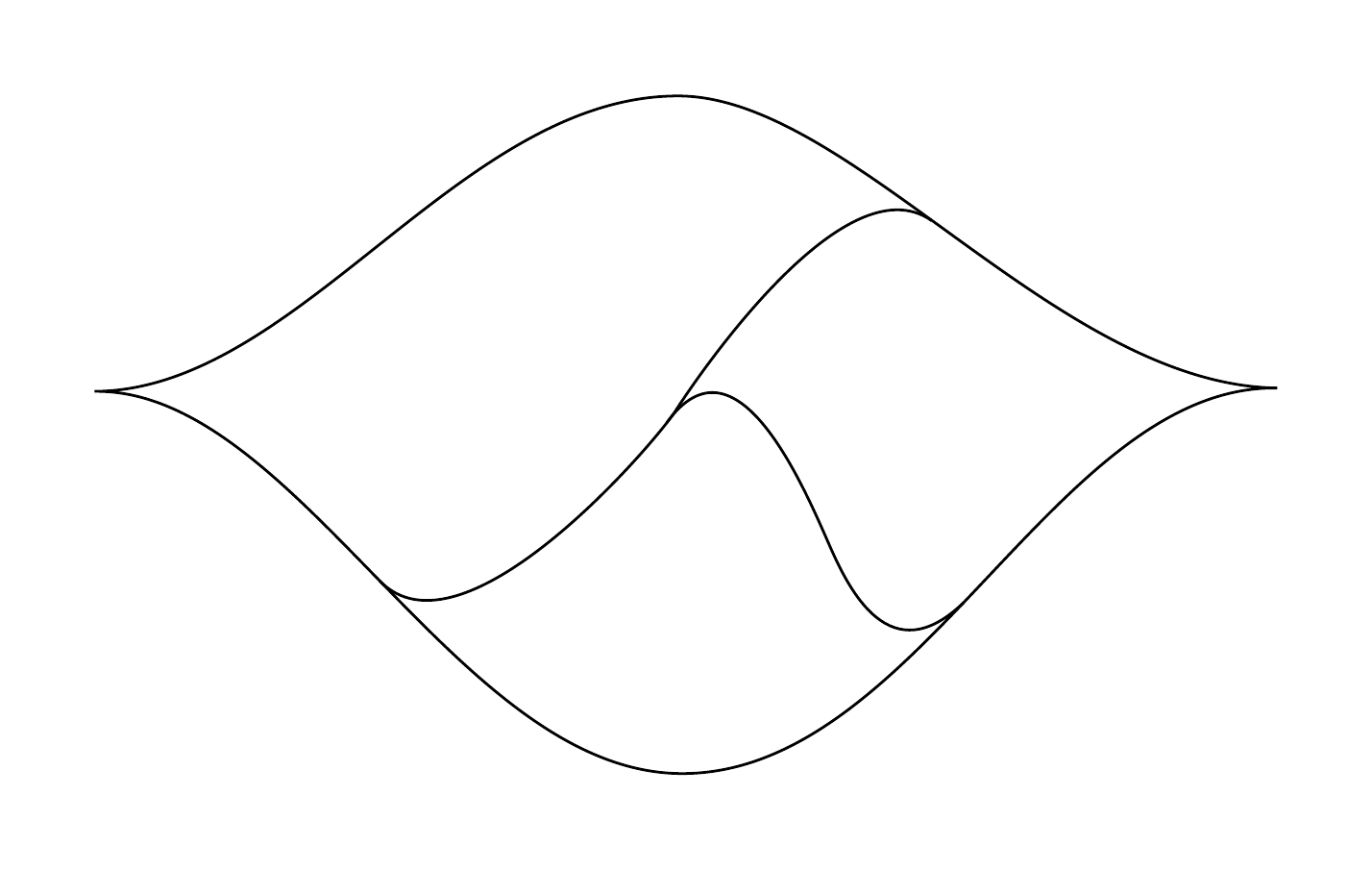} 
\end{center}

More generally, there is an arboreal singularity for each tree.  The ones above correspond to the 
trees $\bullet$, $\bullet \to \bullet$, and $\bullet \to \bullet \to \bullet$.  They are of interest
because they are expected to give in some sense deformation-generic  
models for Legendrian singularities, and for for skeleta of Weinstein manifolds \cite{nadler-noncharacteristic, starkston-arboreal, eliashbergweinsteinrevisited}.

The arboreal singularities are essentially defined as cones on the above pictures.  This has,
at first glance, little in common with what Arnol'd would have associated to the same trees and called 
the $A_n$ singularities --- these being given by the singular fibers of the functions $y^2 - x^{n+1}$.   
Nevertheless, there is a relation.  For example, one can see from the pictures that the homotopy type
of the link of the depicted arboreal singularity is the same as the Milnor fibre of the $A_n$ singularity.  

Let us note another hint that the two objects should be related.  Nadler calculated in \cite{nadlerarboreal} certain 
microlocal sheaf invariants associated to the arboreal singularity.  On the other hand, Seidel  in \cite{seidelbook} 
associates to a function, such as $y^2 - x^{n+1}$, a certain category which can be calculated by his categorification
of Picard-Lefschetz theory.  Both calculations yield the same result: modules over the corresponding tree quiver. 

The purpose of this note is to clarify the geometric relationship between these structures.  First we must abstract from
the singularity theoretic setting only the relevant symplectic geometry.  A small perturbation of the function whose 
singularity we are studying gives a Lefschetz fibration $f: \CC^n \to \CC$.  Recall more abstractly that to a Liouville manifold
$F$ and an ordered collection of Lagrangian spheres $S_1, \ldots, S_k \subset F$, it is possible to form an exact
symplectic Lefschetz fibration $X \to \CC$ with general fibre $F$ and vanishing cycles $S_1, \ldots, S_k$. 
(Roughly speaking, to make $X$ one takes $F$ times a unit disk and attaches a handle along the Legendrian 
$S_r$ in the fiber at $L \times e^{2\pi i r / k}$.)  

In fact, we want to discard the Lefschetz fibration as well, and consider only the pair $(X, F)$.  Technically we ask $X$
to be the completion of a Liouville domain, with a domain completing to $F$ contained in the contact boundary $\partial_\infty X$ of the domain
completing to $X$.  The deformation equivalence class of this pair already
suffices to determine the derived Fukaya-Seidel category of the fibration \cite{sylvanthesis, gpssectorsoc}.\footnote{Note 
that there is a difference between deformation equivalence of Lefschetz fibrations and of Liouville pairs.  This is reflected in the 
difference between given the Fukaya-Seidel category together with a generating exceptional collection up to mutation, 
and just being given the category.  In any case, we consider here deformation equivalence of pairs.}   The geometry of 
such pairs is studied e.g. in \cite{avdek, sylvanthesis, gpssectorsoc, eliashbergweinsteinrevisited}.  
They are a symplecto-geometric version of manifolds with boundary; in particular, the cotangent bundle of a manifold
with boundary  yields such a pair. 

Recall that to a Liouville manifold $(X, \omega = d\lambda)$ one associates the {\em skeleton} $\skull(X)$, defined to be the locus
of points which do not escape under the Liouville flow.  Similarly, to a Liouville pair $(X, F, \omega = d\lambda)$, one associates
the {\em relative skeleton} $\skull(X, F)$, given by the locus of points in $X$ which do not escape to $\partial_\infty X \setminus \skull(F)$ 
under the Liouville flow.  Note that the skeleton and relative skeleton are {\em most certainly not} deformation invariants, though it is
true that $\skull(F)$ is determined by the contact structure on $\partial_\infty X$, rather than a contact form.

For a tree $T$, we write $\Pi_T$ for the plumbing of cotangent bundles of spheres with dual graph $T$.  Given an
ordering of the spheres, we may form the corresponding Lefschetz fibration, $(X, \Pi_T)$.  Note that re-ordering the 
spheres in such a way that intersecting spheres -- adjacent nodes of the tree -- are not interchanged evidently induces
an isotopy of the Lefschetz fibrations.  Given a rooting $\vec{T}$ of the tree $T$, we take any total order compatible with the
partial order induced by the rooting; by the previous remark, these lead to isotopic Lefschetz fibrations.  It is not difficult to see
that the total space of this fibration is just $\RR^{2n}$.  We write $(\RR^{2n}, \Pi_{\vec{T}})$ for this Liouville pair. 
Here we show: 
%\footnote{By ``show'' I mean that by a series of pictures and descriptions in plain English, 
%I will convey to a geometrically minded reader how to go from one of these things to another.  
%The reader who prefers calculations in coordinates is invited to provide their own.}

\begin{theorem} \label{arborlefschetz} 
Let $\vec{T}$ be a rooted tree.  Then $(\RR^{2n}, \Pi_{\vec{T}})$ is deformation equivalent to a Liouville pair
whose relative skeleton is the arboreal singularity associated to $\vec{T}$. 
\end{theorem}

\begin{remark}
As explained in \cite{nadlerarboreal}, the arboreal link admits a cover by arboreal singularities of lower dimension, indexed by correspondences
of trees.  It would be interesting to understand how this cover and these correspondences interact with deformation to a Lefschetz fibration.  
\end{remark}

\vspace{2mm}
{\bf Acknowledgements.}  
I thank 
Roger Casals, Yakov Eliashberg, Sheel Ganatra, Peter Lambert-Cole, Emmy Murphy, John Pardon, and Laura Starkston for helpful discussions and comments on earlier versions of this note. 

\begin{remark}
Some remarks on the history of Thm \ref{arborlefschetz}. 
Paul Seidel noted the relationship between the arboreal link and the Milnor fiber of $A_2, A_3$
singularities after a talk of John Pardon, who then asked me whether such a thing might hold in general.  

Thm \ref{arborlefschetz} was announced in \cite[Rem. 1.6]{gpssectorsoc}.  Originally we planned to use it to develop properties 
of arboreal covers of Weinstein manifolds, both to prove that the Fukaya category cosheafifies over an arboreal skeleton,
and to compare the resulting cosheaf to the one coming from microlocal sheaf theory \cite{kashiwara-schapira, nadler-wrapped, shende-microlocal}.  Since
this time, our strategy to prove the cosheaf property  has evolved so as not to require arborealization 
(see \cite[Thm. 1.20]{gpsstructural} for a representative special case)  and in addition due to \cite{gpstheater} (and \cite{shende-microlocal, nadler-shende})
we no longer require any special form for the skeleton to make the local identification with the sheaf category.  In particular, this result
is no longer necessary for that programme.

Of course, the deformation-genericity of arboreal singularities means they are of interest  far beyond their role in categorical 
calculations.  Perhaps the above result may clarify the nature of these objects.  At the least it decreases the cognitive dissonance
caused by the fact that after \cite{nadlerarboreal}, the term ``$A_n$ singularity'' acquired more than one meaning. 
\end{remark}

\section{Some singular Legendrians}

We typically write $\vec{T}$ to mean a rooted tree.  
If $v$ is any vertex of the tree, we write $\vec{T}(v)$ for the sub-tree growing from $v$ (trees grow away from the root).  

By a {\em shrub} we mean a tree with all vertices at distance at most one from the root.  We write 
$\vec{T}_{\le 1}(v)$ for the shrub growing from a given vertex.

By a singular Legendrian, we mean a finite union of isotropic submanifolds which is the closure of its smooth Legendrian
locus.   Here we present three different families of singular Legendrians associated to rooted trees; we define them inductively 
just by drawing fronts.

\subsection{Arboreal singularities}

\begin{center}
\begin{figure}
\begin{tabular}{cccc}
\includegraphics[scale=0.25]{A1} & \includegraphics[scale=0.25]{A2} & \includegraphics[scale=0.25]{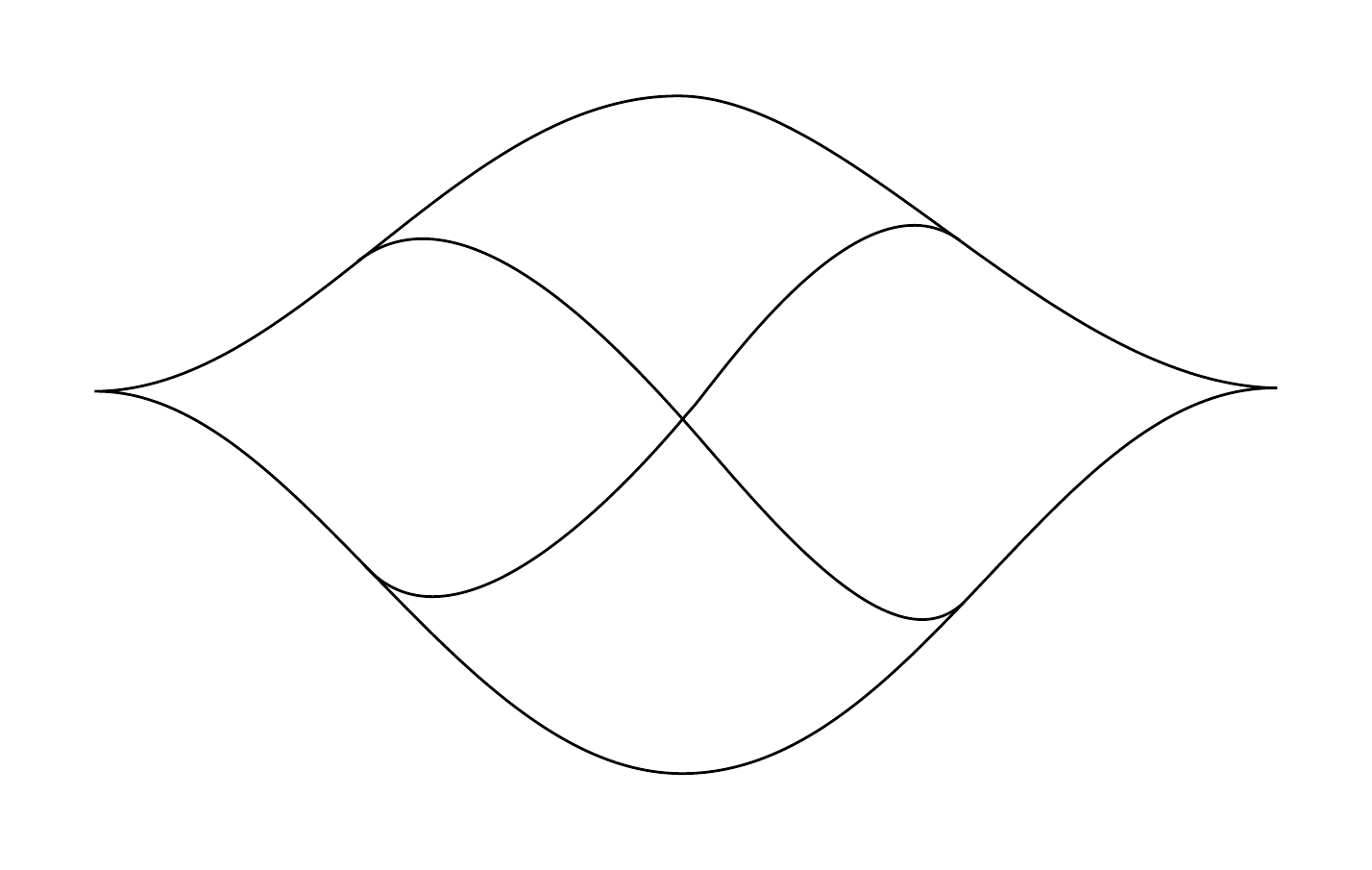} & \includegraphics[scale=0.25]{A3} \\
$\bullet$ & $\bullet \to \bullet$ & $\bullet \leftarrow \bullet \to \bullet$ & $\bullet \to \bullet \to \bullet$
\end{tabular}
\caption{\label{arblinks} Links of the arboreal singularities for various trees.}
\end{figure}
\end{center}
\vspace{-1cm}

\begin{definition}
Let $\vec{T}$ be a rooted tree, and fix $n \ge |\vec{T}| -1 $.  We will define the {\em arboreal link} corresponding to $\vec{T}$
as a certain explicit Legendrian by giving a front projection to $\RR^n$.  The construction is recursive in
nature.  We assume $n \ge 2$.  

We take the first $|\vec{T}|-1$ coordinates of $\RR^n$ to be indexed by the non-root vertices of the tree.  In front projections,
one direction is distinguished (``no vertical tangencies''); we take the vertical direction to be given by the sum of the coordinates. 

Everything is built from the front associated to $\bullet \to \bullet$ in Figure \ref{arblinks}. 
Note that the diagonal line dividing the big unknot is, away from the big unknot, a disk in a coordinate axis (recall that our coordinate hyperplanes
are always slanted).  The choice of coordinates is such that this axis is the one on which the leaf vertex coordinate vanishes.

For shrubs the construction is as follows.  Take the front for $\bullet \to \bullet$, and suspend
it appropriately so it becomes a front in the desired dimension.  Now it is a big unknot (saucer) with a slice through the middle, which is
mostly just a disk in a coordinate hyperplane.  By permuting coordinates, this could be any coordinate hyperplane.  The desired front
is obtained by taking the union the fronts corresponding to the hyperplanes named by the non-root vertices.  For an example when $n=2$, 
see the front associated to $\bullet \leftarrow \bullet \to \bullet$ in Figure \ref{arblinks}.

In general we proceed as follows.  First apply the above construction to the shrub obtained from pruning all vertices 
of distance $>1$ from the root.  Now, for each vertex $v$ which is one away from the root, re-focus attention on the unknot bounded by 
(the lower) half the original unknot plus the hypersurface associated to $v$.   Working inside this new unknot, apply the algorithm to the tree 
$\vec T(v)$ which grows from $v$.

The careful reader may have noticed that the unknot we used for the second and later steps of the algorithm is singular -- it has a corner 
where the middle piece meets the original unknot.  Said reader may convince themselves that this leads to no ambiguity in the 
description.  One possibility is to make even the original unknot singular, e.g. by drawing the (non-generic) front as a cube stood
on its corner.   $\triangle$
\end{definition}

\begin{remark}
Nadler originally introduced the arboreal singularities as Legendrian singularities.  We have described
the Legendrian link of the Lagrangian projection of this entity.  
\end{remark}

\begin{figure}
\begin{center}
\includegraphics[scale=0.5]{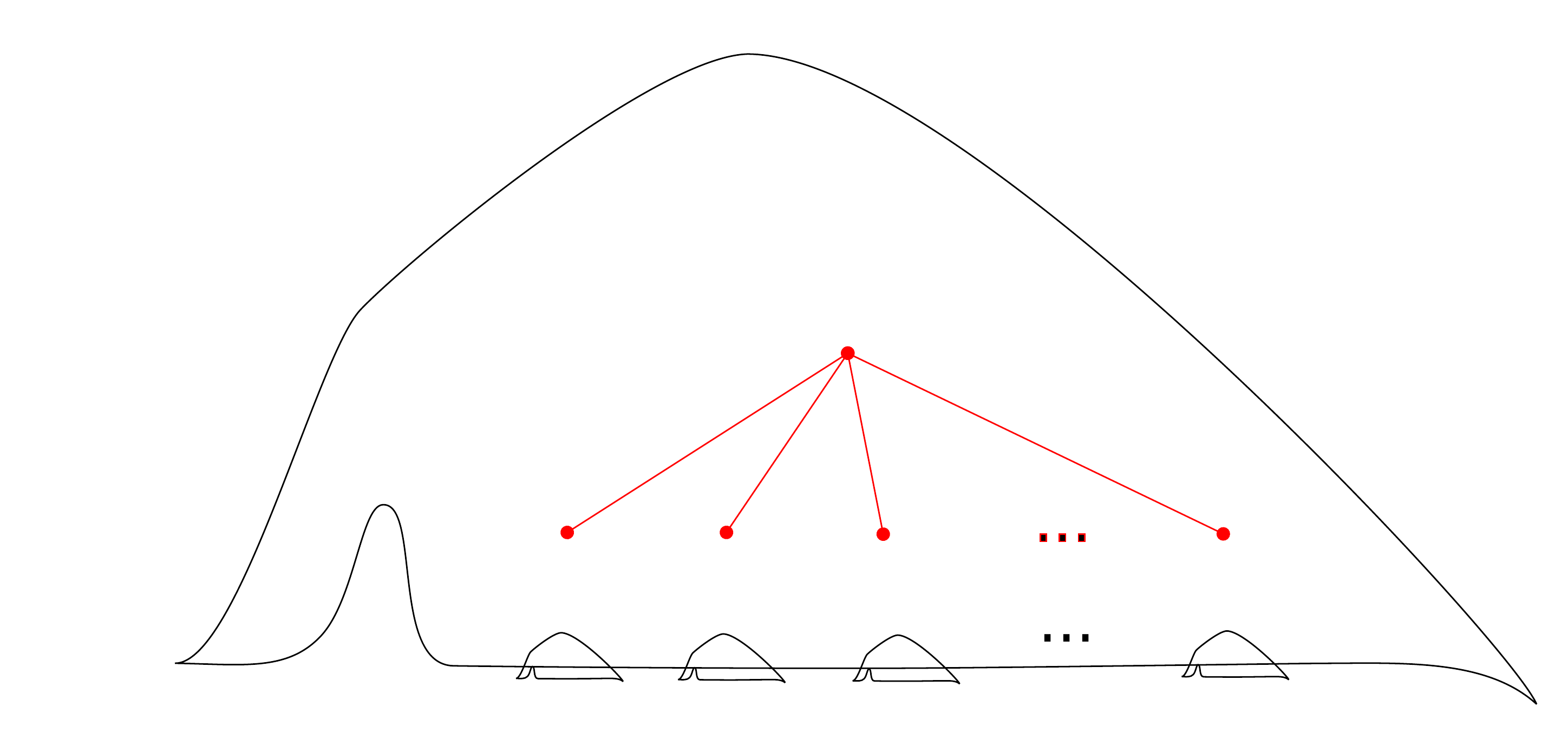}
\caption{\label{armadillo} The armadillo of a shrub.}
\end{center}
\end{figure}

\subsection{Armadillos}

\begin{definition} \label{plumbfront}
Let $\vec{T}$ be a rooted tree.  Extend the partial order on the vertices of $\vec{T}$ coming from the tree structure to a total order.  From this
data we define a front recursively as follows.  

To the shrub growing from the root, we associate the picture in Figure \ref{armadillo}. 
Here the smaller unknots are ordered left to right matching the total ordering of the vertices.  The full picture associated to $\vec T$ is 
given by now recursively applying this construction to the trees which grow from the vertices at distance one from the root, except
now using the depicted small unknots as the big unknot.  $\triangle$
\end{definition}

\begin{remark}
The picture is in a 2d front plane, but evidently this prescription makes sense for any number of dimensions
$n \ge 2$.  Note that (up to a non-contractible space of isotopies) the total ordering of the vertices of $\vec T$ becomes irrelevant in
higher dimension. 
\end{remark}

\subsection{Motherships}

It will be convenient to interpolate between the arboreal singularities and the armadillos by introducing another class of singular Legendrians
indexed by trees.  As before, we give an inductive definition. 

\begin{definition} \label{motherfront} {\bf (motherships)} 
Let $\vec T$ be a rooted tree.  Extend the partial order on the vertices of $\vec T$ coming from the tree structure to a total order.  From this
data we define a front recursively as follows.  

To the shrub growing from the root, we associate the picture in Figure \ref{mothership}.
Here the smaller unknots are ordered left to right matching the total ordering of the vertices.  The full picture associated to $\vec{T}$ is 
given by now recursively applying this construction to the trees which grow from the vertices at distance one from the root, except
now using the depicted small unknots as the big unknot.  $\triangle$
\end{definition}

\begin{center}
\begin{figure}
\includegraphics[scale=0.5]{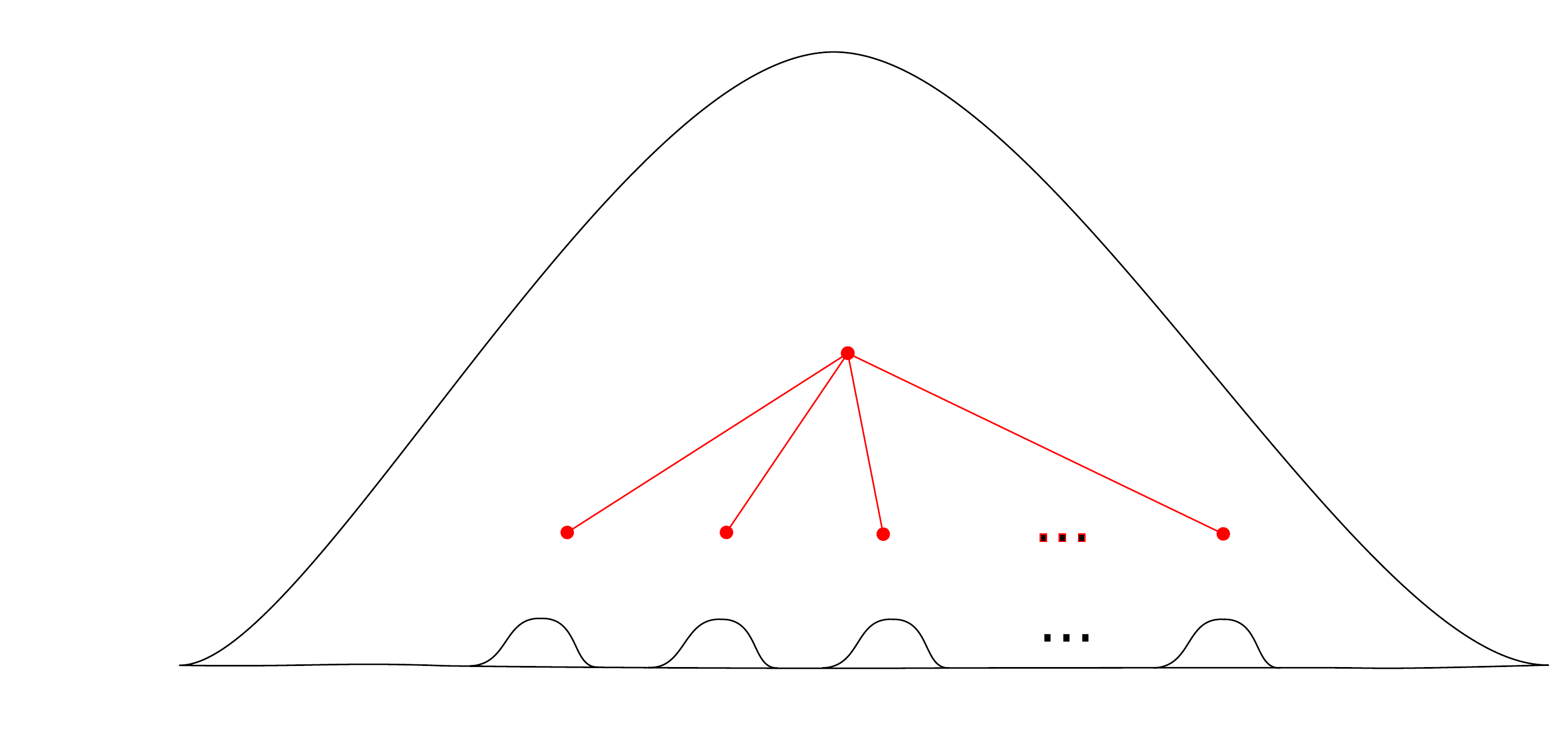}
\caption{\label{mothership} The mothership of a shrub.}
\end{figure}
\end{center}

\vspace{-1cm}

\begin{remark}
In order that the big unknot have exactly the same front as the small unknots, it must be taken to be singular as a Legendrian
above its cusps.  
\end{remark}

The total ordering of branches becomes irrelevant in front dimension $\ge 2$.  
In higher dimensions, it 
will be convenient to have a variant where the loci of attaching smaller saucers is prescribed in a way
more similar to the description of the arboreal link.

\begin{definition} {(\bf coordinated motherships)}
Let $\vec T$ be a rooted tree, and choose some $n \ge |\vec T| - 1$.  We will draw a front in $\RR^n$.  As for the arboreal singularity, 
non-root vertices of the tree index coordinates of $\RR^n$ (excess coordinates go unindexed) and the vertical direction
is the sum of the coordinates. 

Begin with the front of a flying saucer, centered about the origin.  The bottom half of the saucer projects vertically to a disk of some fixed
radius in the plane $\sum x_i = 0$; correspondingly I will use the $x_i$ as coordinates on this bottom half (subject to the condition that
they sum to zero).  Inscribe a simplex of the same dimension in this disk, with the facets being given by setting some coordinate to 
a constant value.  Mark a point at the center of each facet.   

In short, there are now $n$ marked points on the bottom of the flying saucer, one for each coordinate.   For each of these which corresponds
to a vertex of the tree with distance one from the root, take a small disk around it, and use it as the base for a (singular) saucer.  (The top 
should be a disk of the same size).  

Now recursively apply this algorithm for the trees growing from the vertices at distance one for the root, in each case replacing the original
big saucer with the small one that was created in the previous step. $\triangle$
\end{definition}

\begin{figure}
\begin{center}
\begin{tabular}{ccc}
$\bullet \to \bullet$ & $\bullet \to \bullet \to \bullet$ & $\bullet \leftarrow \bullet \to \bullet$ \\
\includegraphics[scale=0.3]{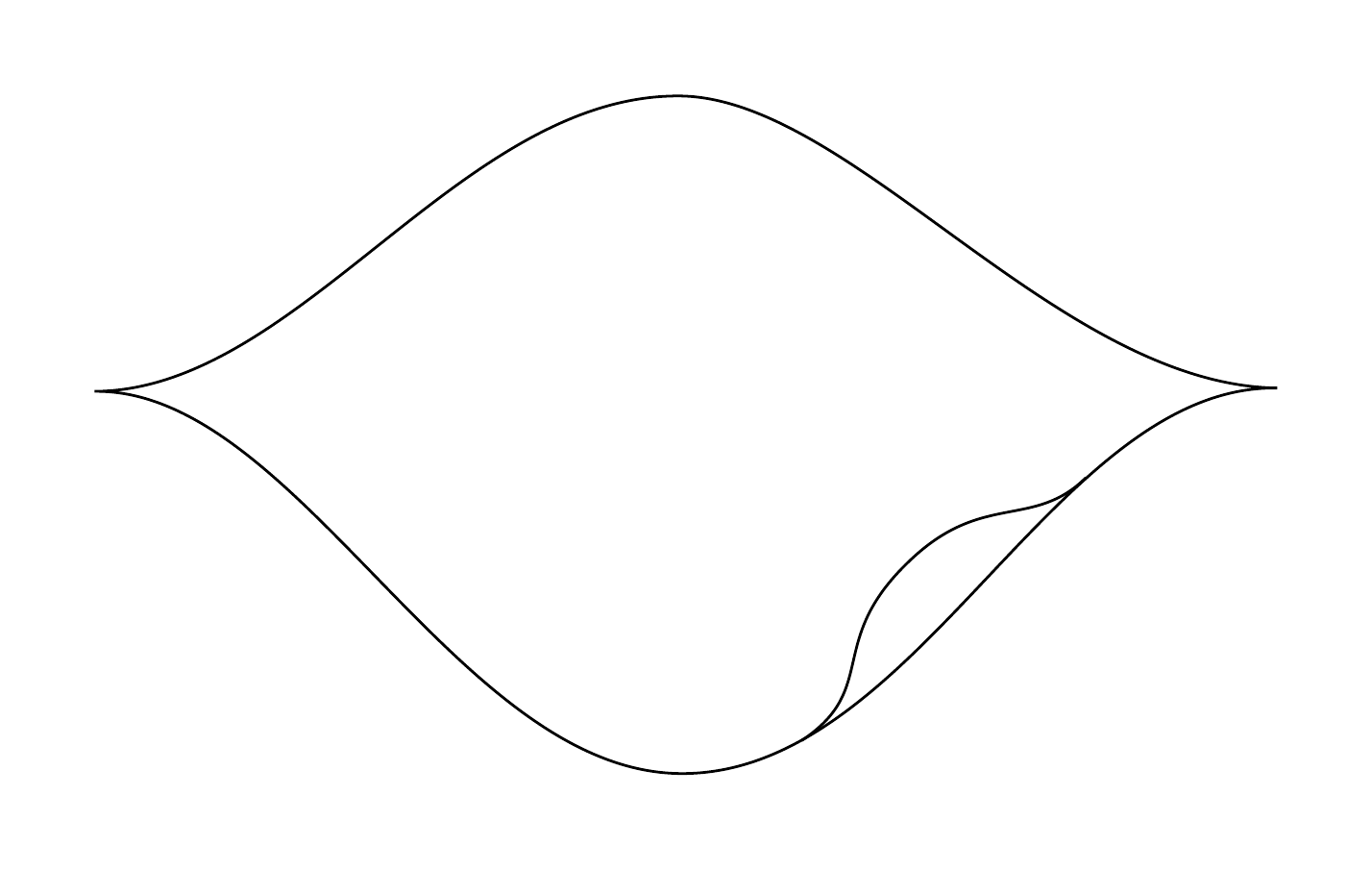} & \includegraphics[scale=0.3]{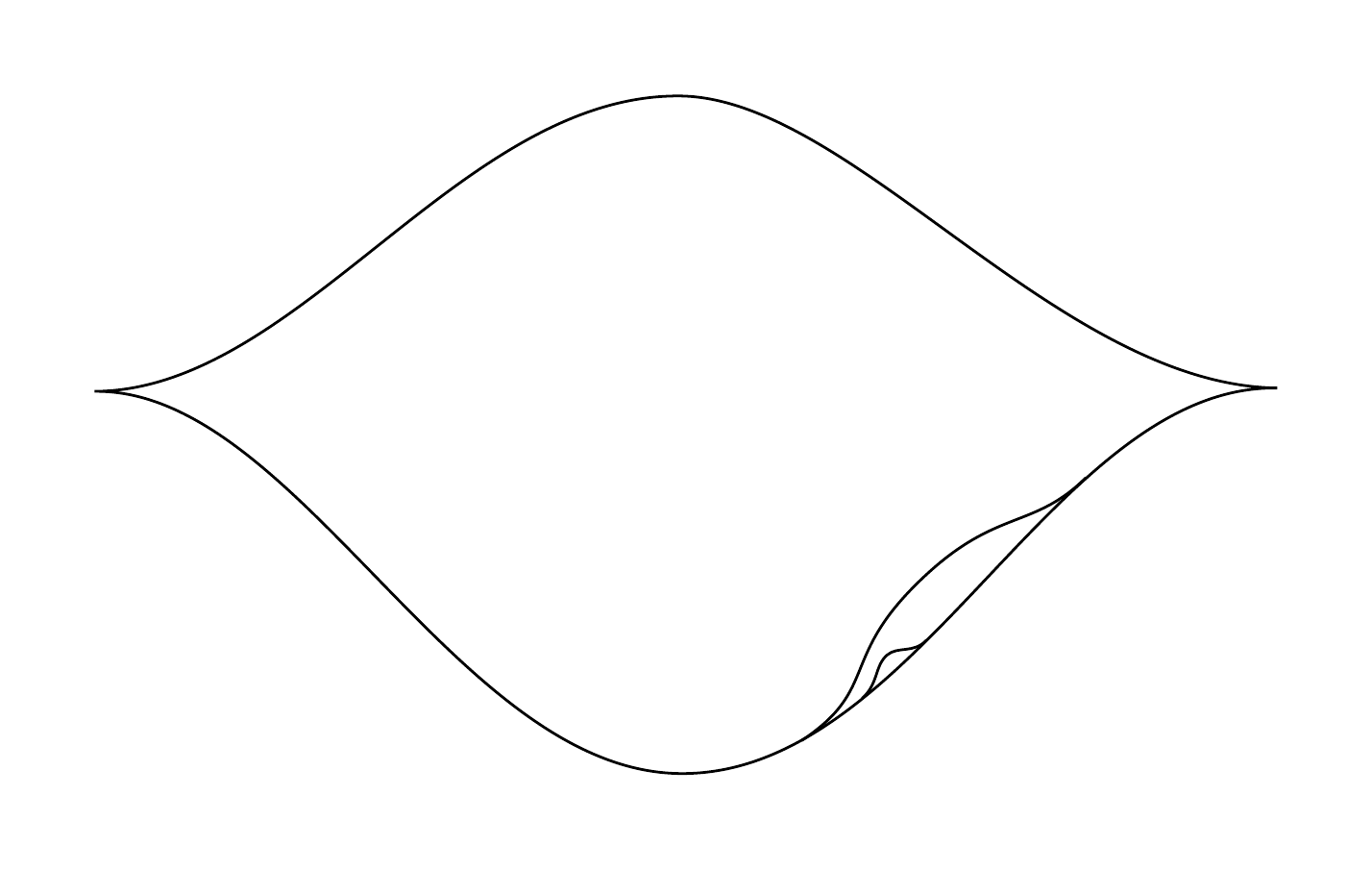} & \includegraphics[scale=0.3]{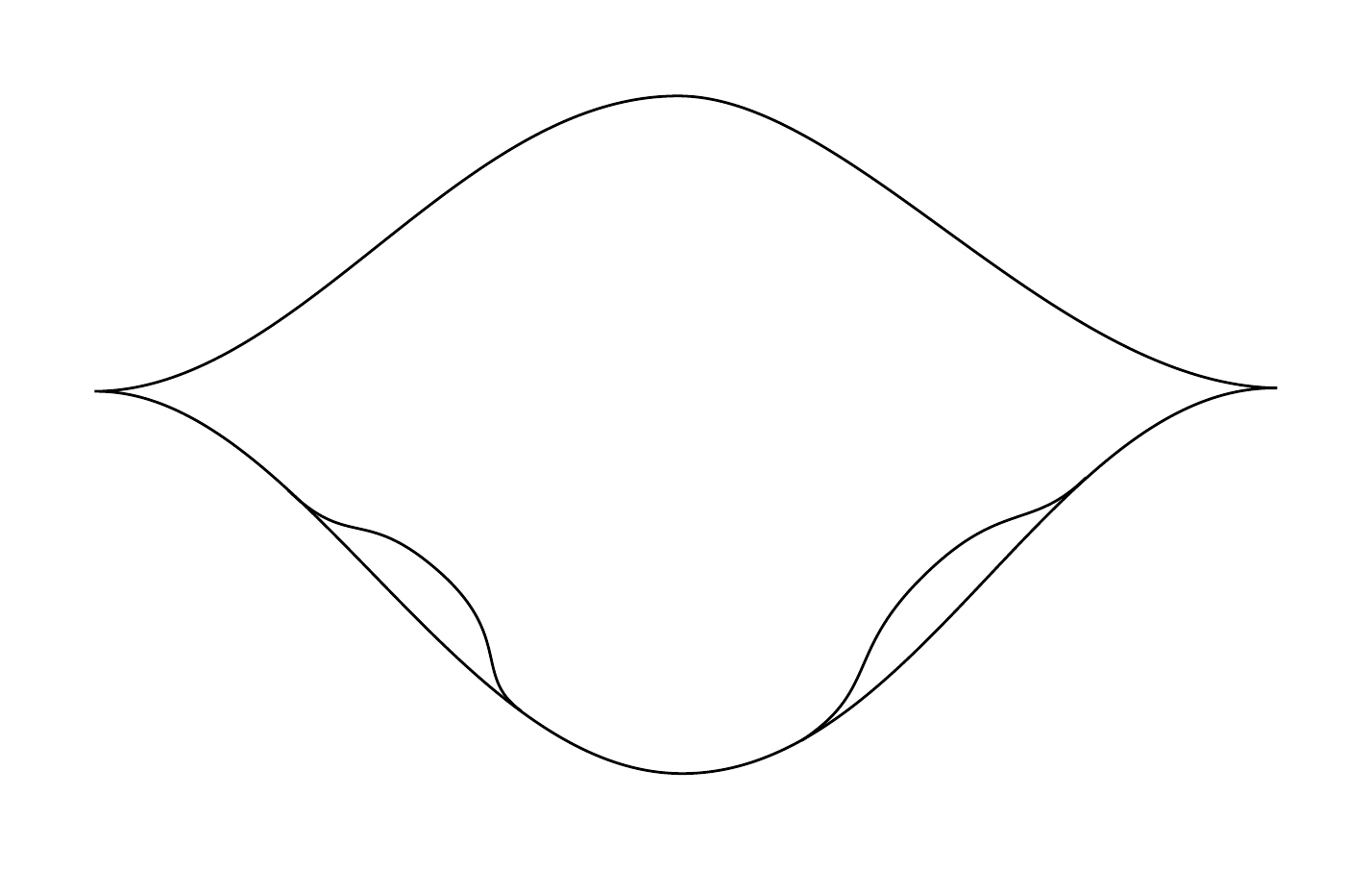}
\end{tabular}
\end{center}
\caption{\label{coordmother} Coordinated motherships for $n=2$.}
\end{figure}

\section{Ribbons}

\subsection{Generalities}

\begin{definition}
Let $V$ be contact and $\Lambda \subset V$ be a singular Legendrian.  A {\em ribbon} for $\Lambda$ is a
codimension one submanifold $R \subset V$ such that some local contact form near $\lambda$ determines 
a Liouville structure on $R$ with skeleton $\Lambda$. 
\end{definition}

We have the following standard facts: 

\begin{lemma}
The Reeb vector field is transverse to any ribbon
\end{lemma}
\begin{proof}
Recall by definition the Reeb vector field is in the kernel of $d \lambda$.  
Since $d\lambda|_R$ is nondegenerate, the Reeb field cannot be tangent to $R$ at any point,
thus is transverse to $R$.  
\end{proof}

\begin{lemma}
A ribbon determines a contact embedding of (an neighborhood of $R$ in the) contactization $R \times \RR \to V$ carrying the skeleton of $R$ to $\LL$.  
\end{lemma}
\begin{proof}
Pushing by Reeb sweeps out the desired embedding. 
\end{proof}

We {\em do not} know whether any two ribbons are isotopic.  It is however possible to show the following, 
which already implies that no Floer theoretic invariants will depend on the choice of the ribbon. 

\begin{lemma}
Given any two ribbons $R, S$ for the same Legendrian, one can find ribbons $R', R''$ isotopic as ribbons
to $R$, such that  
$R' \subset S \subset R''$ and the inclusion $R' \subset R''$ is trivial (i.e. the Liouville flow on $R''$ gives an isotopy between them).  
\end{lemma}
\begin{proof}
The point is that $R$ and $S$ must both have tangent spaces transverse to the Reeb flow along the skeleton, hence 
in a small neighborhood thereof, each will be graphical over the other in the aforementioned embedding of the contactization. 
\end{proof}

\begin{remark}
We say that the ribbon of a singular Legendrian is unique up to matryoshka. 
\end{remark}

A fundamental question is: 

\begin{question}
Which singular Legendrians admit a ribbon?
\end{question}

There are evident local obstructions (an example of John Pardon:  take many smooth Legendrian curves 
with varying second order behavior through a point; no surface can contain them all).  We do not know
whether there are global obstructions.  

The situation is of course even worse for the family version: 

\begin{question}
When does a family of singular Legendrians arise as the family of cores of an isotopy of ribbons?
\end{question}

\begin{definition}
We say a 1-parameter family of singular Legendrians which arises as the family of cores of an isotopy of ribbons
is a {\em ribbotopy}.\footnote{It moves the ribs of the skeleton.}
\end{definition}

Given a ribbon, two things we can do to construct a family of ribbons are the following.  
One is to apply an ambient contact isotopy.  The other is to apply a contact contact isotopy along
a contact level of the ribbon itself: 

\begin{lemma} \cite{cieliebakeliashberg} 
Let $(R, \lambda)$ be a Liouville domain and $R^{in} \subset R$ a subdomain.  
Then from a contact isotopy $\phi_t: \partial R^{in} \to \partial R^{in}$ one can construct a 1-parameter family $\lambda_t$ of Liouville forms
such that
the Liouville flow is unchanged away from a collar neighborhood of $\partial R^{in}$, and integrates to $\phi_t$ when 
traveling across this neighborhood. 
\end{lemma}

The ribbotopies we use will be of the following form.  Given some contact manifold $(V, \lambda)$ and Liouville hypersurface $R$, 
we will cut $R$ into $R^{in}$ and $R^{out} = R \setminus R^{in}$.   We push $R^{out}$ by an ambient contact isotopy which restricts
to a contact isotopy along $\partial R^{in}$.  That is, from the point of view of the 1-form on the family of hypersurfaces thusly created, 
the isotopy looks like that of the above lemma.   (Below, the $R^{out}$ will always be the part corresponding what is 
further towards the leaves of the tree.)

\subsection{Graph ribbotopies}

The notions of ribbon and ribbon equivalence are well studied in the context of Legendrian graphs in contact 3-manifolds. 
Many explicit diagrams of isotopies and ribbotopies
can be found e.g. in the papers \cite{baader-ishikawa, odonnol-pavalescu, lambertcole-odonnol}.  
In Figure \ref{graphribbotopies} we collect the graph ribbotopies.  They
allow a sanity check on the later alleged ribbotopies by drawing them in this dimension as a sequence of these moves. 
In fact, except for the move we require in Section \ref{atom}, all other ribbotopies we use can be obtained as stabilizations of
these moves. 

\begin{figure}
\begin{center}
\includegraphics[scale=0.8]{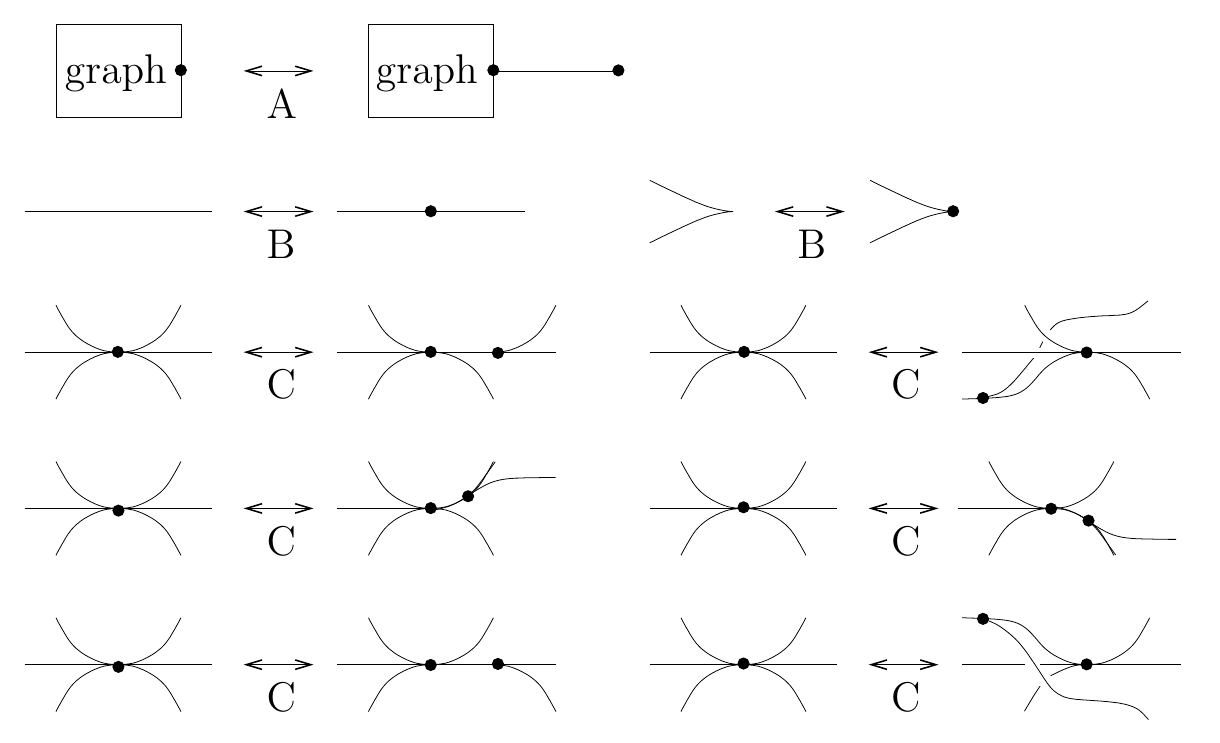}
\end{center}
\caption{\label{graphribbotopies} Graph ribbotopies  (from \cite{baader-ishikawa}).}
\end{figure}

\begin{figure}
\begin{center} 
\includegraphics[scale=0.8]{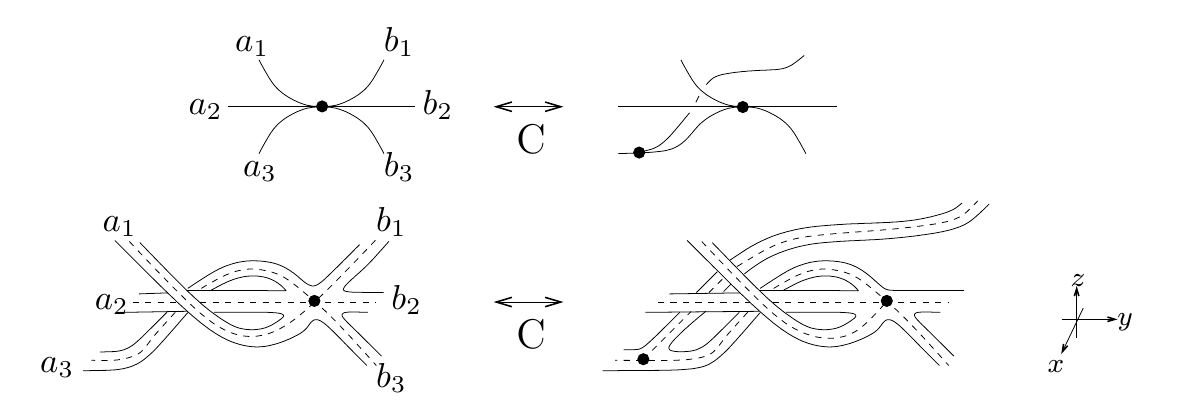}
\caption{\label{graphribbotopies} The `C' moves are a bit weird from the point of view of the front projection.  It is clearer when the surface is drawn. 
  (from \cite{baader-ishikawa}).}
\end{center}
\end{figure}

\begin{figure}
\begin{center}
\includegraphics{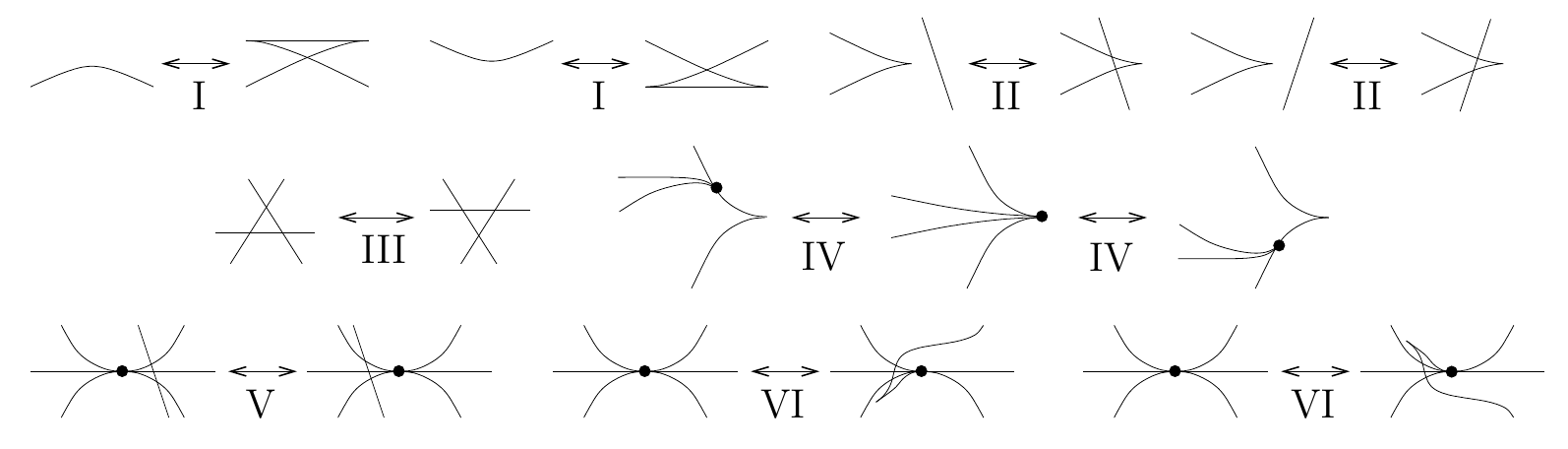} 
\end{center}
\caption{\label{reidemeister} The Reidemeister moves for graphs (from \cite{baader-ishikawa}).}
\end{figure}

It is useful to note that a Legendrian graph has, at each vertex, a canonical cyclic order of the edges, since all their tangents
must lie in the contact plane.  What is going on in the `C' moves is just that a given edge is sliding onto an edge either immediately
before or immediately after it in the cyclic order.  The apparent 
weirdness of the `C' moves has to do with the fact that the cyclic order is messed up by the front projection.  In the front projection, the 
cyclic order is:  negative-to-positive slopes to the right of the vertex, then positive-to-negative slopes to the left of the vertex. 

Also useful will be the Reidemeister moves for Legendrian graphs, i.e., ways the front projection can be altered
by a small contact isotopy.  These are again from \cite{baader-ishikawa}.  
We will frequently draw move VI; a good way to think of it is as half of a Reidemeister 1 move.

\section{Proof of Theorem \ref{arborlefschetz}}

We will show in Section \ref{plumbarms} that the plumbings have link given by the armadillo; in Section \ref{armmo}
that the armadillo is ribbotopic to the mothership, and in Section \ref{atom} 
that the arboreal link is ribbotopic to the coordinated mothership.  The coordinated mothership being obviously ribbotopic to the 
mothership, this will complete the proof of Theorem \ref{arborlefschetz}.

\subsection{The relative skeleton of $(\RR^{2n}, \Pi_{\vec{T}})$} \label{plumbarms}

In \cite{casals-murphy}, an algorithm is given for drawing front projections of 
Legendrians which live in the boundaries of Lefschetz fibrations whose fibre is a sphere plumbing.  We only need the zeroeth 
step of this algorithm: drawing the front projection of the skeleton of the plumbing itself.  In \cite{casals-murphy}, the Legendrians of interest
are in the contact manifold $\partial (\Pi_{\vec{T}} \times \CC)$.  This contact manifold is obtained from $S^{2n-1}$ by attaching
(subcritical) handles along a certain amount of $S^{n-2}$.  Thus they draw fronts in an $\RR^{n}$ with a certain amount of 
$S^{n-2}$ shaped wormholes (one for each sphere of the plumbing).  They give explicitly a picture of the skeleton of the plumbing
in this space.   See Figure \ref{casalsmurphy}. 

\begin{figure}
\begin{center}
\includegraphics{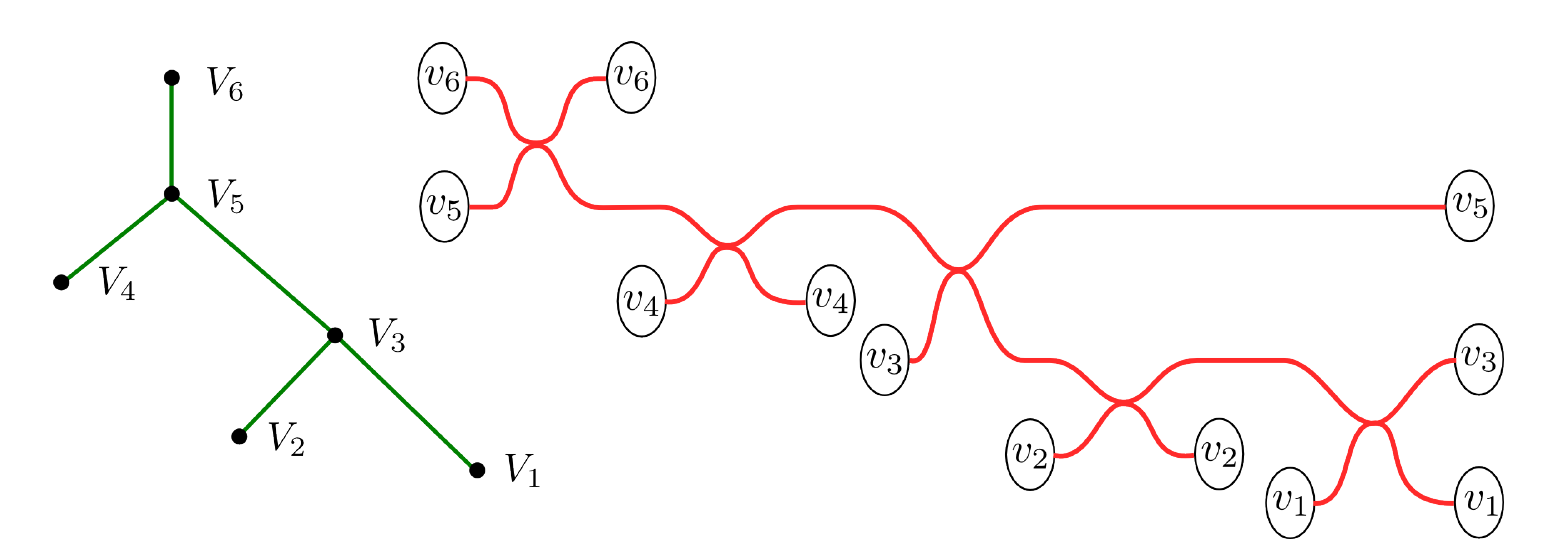}
\end{center}
\caption{\label{casalsmurphy}The front of a plumbing, as drawn in the front space for $\RR^{2n}$ after attaching several subcritical handles.  (From 
\cite{casals-murphy}.)}
\end{figure}

We are interested in understanding this skeleton in a different space: the contact boundary of the Weinstein manifold 
which results from cancelling these $S^{n-2}$'s with critical handles, which should be attached along the $S^{n-1}$'s which are
plumbed together to make the skeleton.  The front projection of the result is given by erasing
the wormholes, adding a small (say upwards) pushoff of each of these $D^{n-1}$'s which are visible in the projection, 
then connecting them to the one below to make a flying saucer.  The result is isotopic to the armadillo of Def. \ref{plumbfront}.   
See Figure \ref{cancelled}. We conclude:

\begin{figure}
\begin{center}
\includegraphics[scale = 0.5]{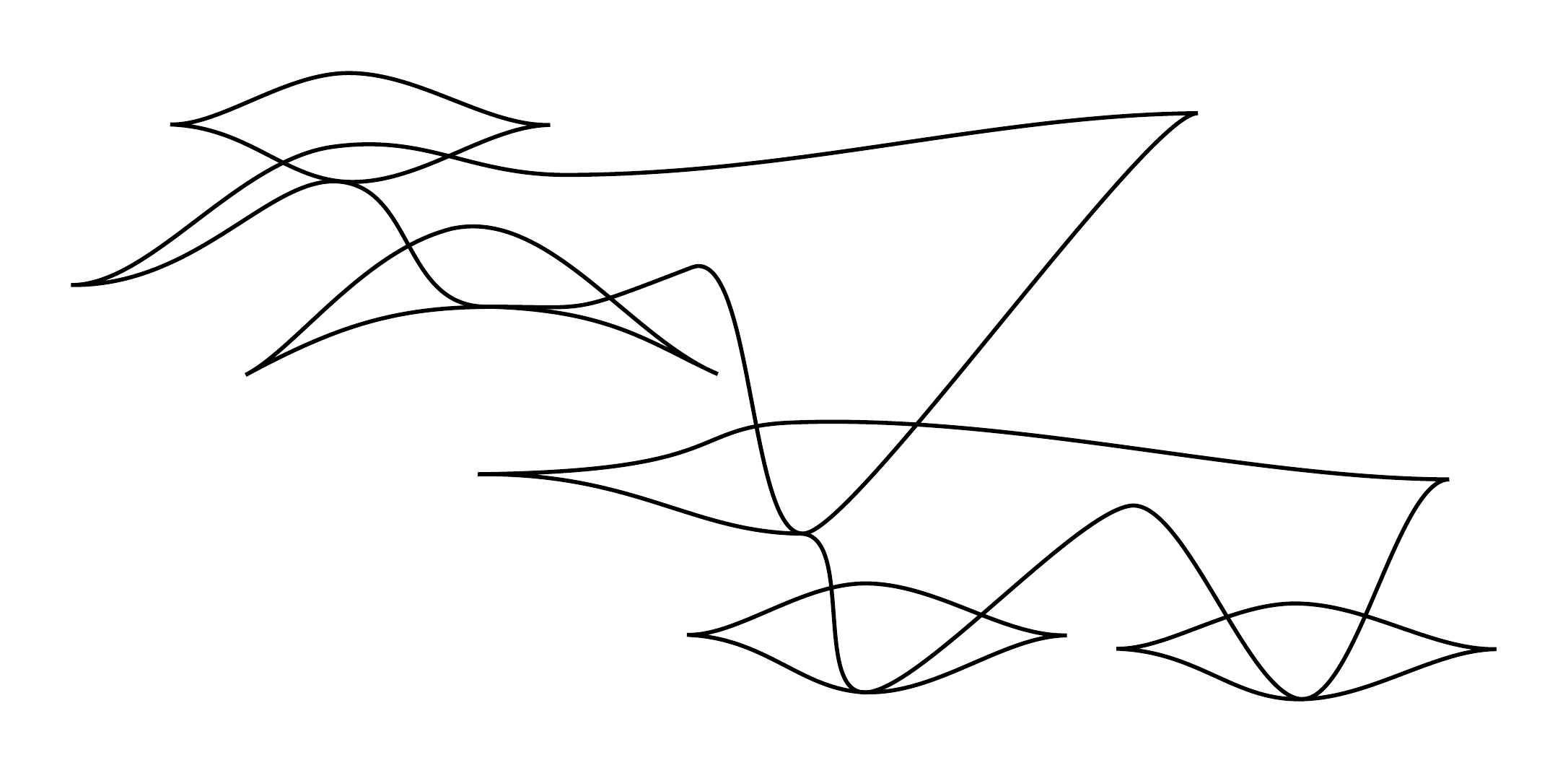}
\end{center}
\caption{\label{cancelled}The front of a plumbing, as drawn in the front space for $\RR^{2n}$.}  
\end{figure}

\begin{proposition}   
Let $\vec{T}$ be a rooted tree.  
Consider the Liouville pair $(\RR^{2n}, \Pi_{\vec{T}})$ arising from the Lefschetz fibration with fibre the plumbing
of sphere cotangent bundles $\Pi_{\vec{T}}$ and vanishing cycles the ordered zero sections.  Then $\skull(\RR^{2n}, \Pi_{\vec{T}})$ 
is carried by an ambient contact isotopy to the Legendrian with front as in Def. \ref{plumbfront}. 
\end{proposition}

\subsection{Armadillos to motherships} \label{armmo}

The basic move to turn an armadillo into a mothership is

\begin{center}
\includegraphics[scale=0.4]{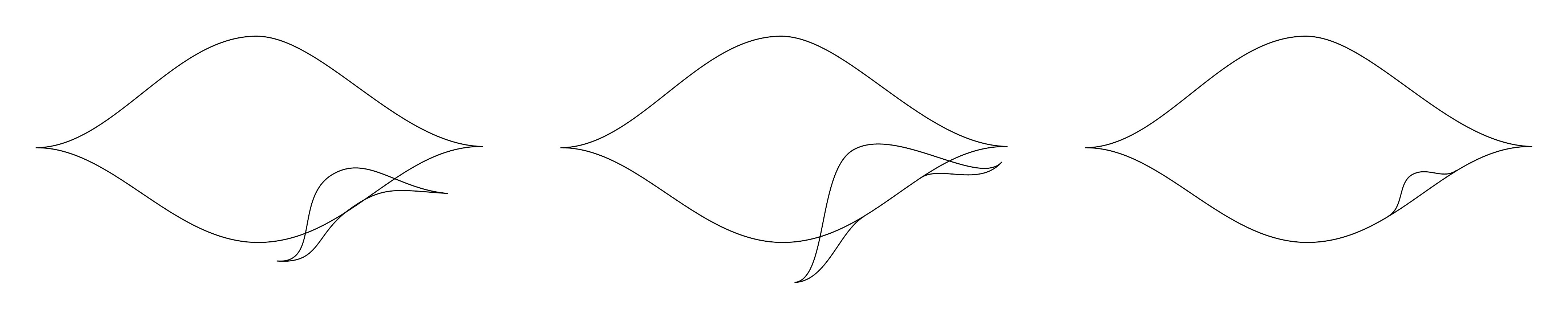}
\end{center}

The first step is a nontrivial ribbotopy, the second is just an isotopy (in this dimension, the first step is the type `C' ribbotopy in the list above,
and the second is the type `VI' Reidemeister move. Evidently this works for shrubs, in any dimension. 

\begin{example}
For $\bullet \leftarrow \bullet \to \bullet$: 
\begin{center}
\includegraphics[scale=0.25]{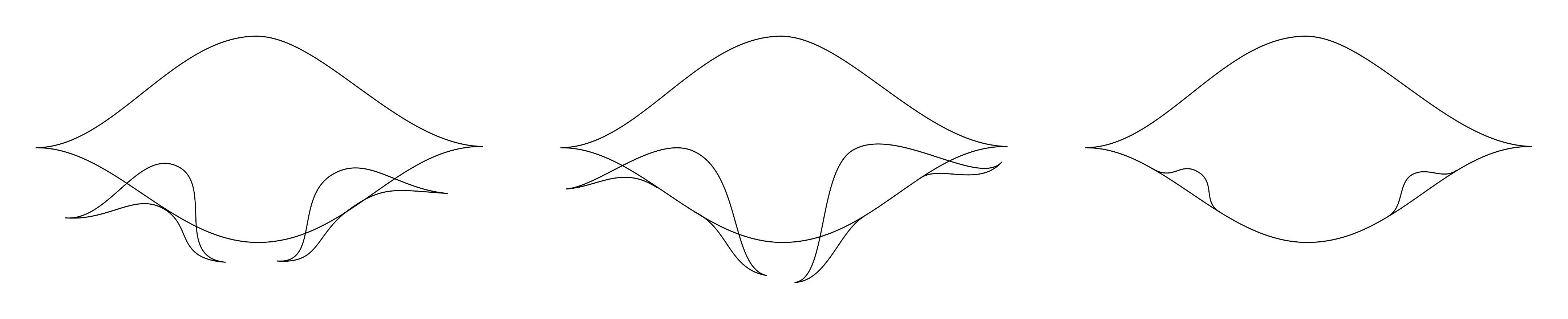}
\end{center}
\end{example}

To set up the inductive procedure one runs into the following difficulty.  Try to apply the above move
starting from the mothership for $\bullet \to \bullet \to \bullet$.  One arrives here: 

\begin{center}
\includegraphics[scale=1]{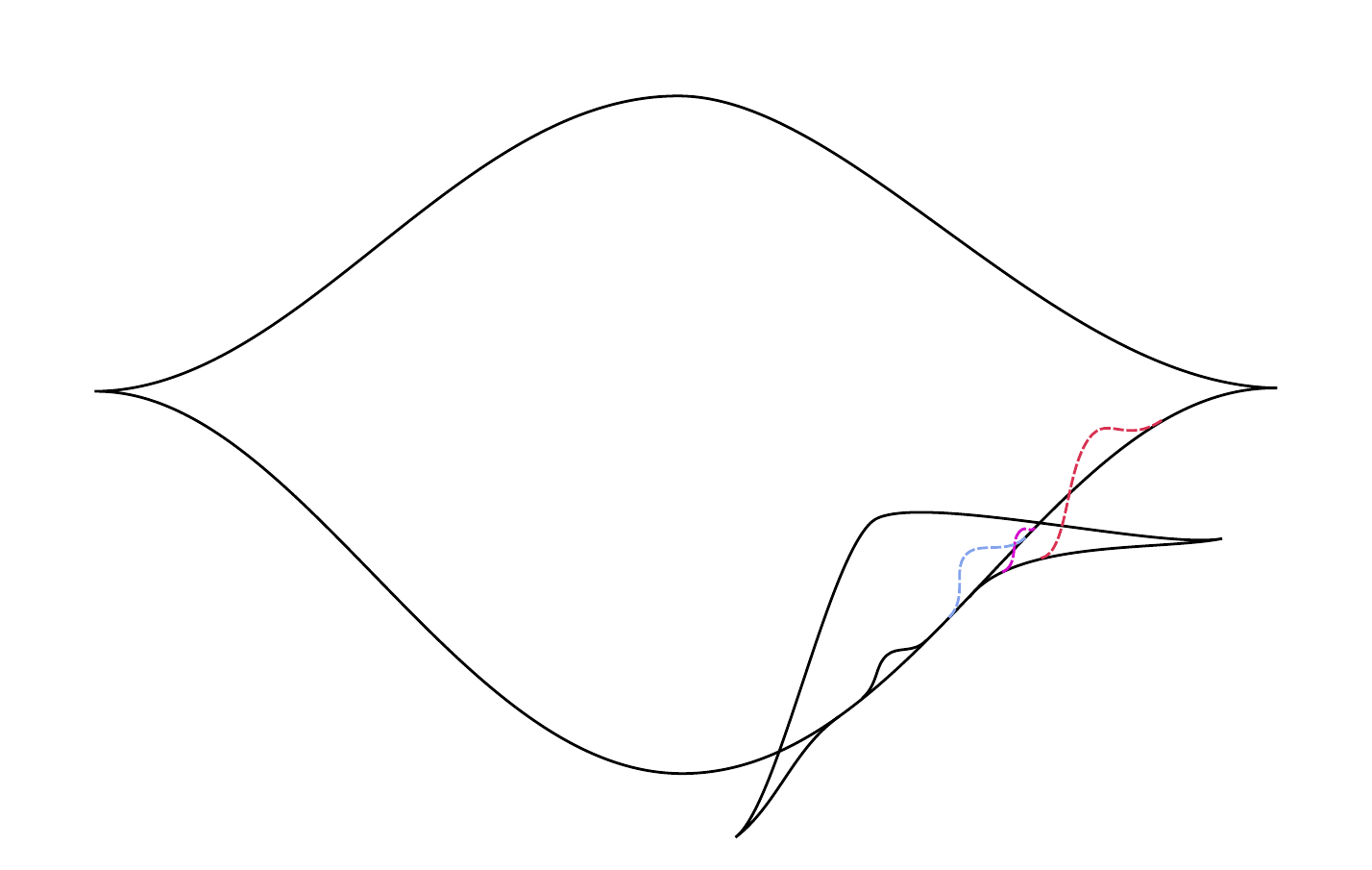}
\end{center}

That is, the smallest unknot is preventing the middle sized unknot from contracting the edge it shares with the largest unknot.  
One can try various things, e.g. sliding the small unknot out of the way as indicated in the above figure.  But then one is faced with the problem of getting
it back where it's supposed to be, namely near the bottom cusps of the middle-sized unknot. 

One thing which does work is the following: 

\begin{center}
\includegraphics[scale=0.25]{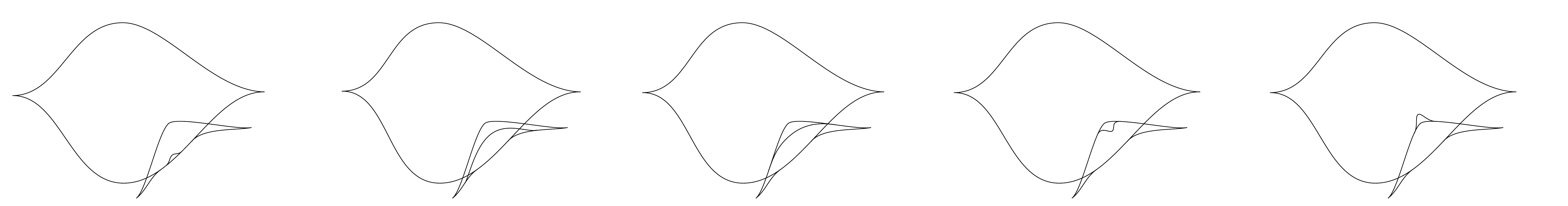}
\end{center}

In the first four pictures, the smaller unknot is travelling along the medium unknot; the junction where they meet is moving by the cone on the Reeb flow
nearby.  The passage from the first picture to the second is a ribbotopy, and from the second to the fourth, just isotopy.  The fourth to the fifth is a 
ribbotopy: the bottom piece of the smallest unknot becomes part of the medium unknot, and the top piece of the smallest unknot (previously
also a piece of the medium unknot) moves up. 

Note this works in any dimension, and with any number of smallest nested unknots.   Indeed, consider where
they attach to the medium unknot. The dynamics of these spheres is that of the fronts of waves emitted simultaneously from several points on the sphere.
The legendrians of these waves never meet each other (after all they are all flowing by Reeb), and eventually the reconverge at the antipodal points. 

In the case that there are further levels of nesting, this procedure should be applied to the lowest (further from the root) nested level first.  Now the smaller levels have gotten out of the way, it is possible to ribbotope the unknots corresponding to the nodes one away in the tree to look like the plumbing model.  

Note there is an ambient isotopy carrying the higher nested unknots living at the top back down to the bottom.  
Here is a picture in front projections, communicated to me  by Peter Lambert-Cole.  The dashed lines indicate 
that a VI move is about to (or has just) been performed, and they are the ``horizontal'' lines with respect to this move. 

\begin{center}
\includegraphics[scale=0.25]{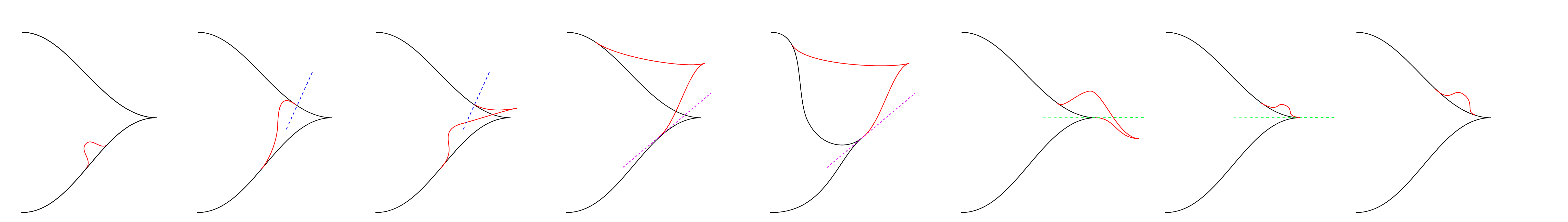}
\end{center}

\subsection{Arboreal links to motherships} \label{atom}

The idea is as follows: each hypersurface we have introduced in making the front diagram for the arboreal link 
is mostly a coordinate hypersurface.  We slide all these hypersurfaces simultaneously along their normal direction
which points downward (recall that the vertical axis is the sum of the coordinates).  A hypersurface corresponding to
a vertex on the $k$'th level of the tree should slide at speed $c^{-k}$ for some constant $c > 1$.  

To define it more precisely, we proceed as usual to first define it for the tree $\bullet \to \bullet$, then extend to shrubs, 
and finally to extend to all trees by induction.  The picture for $\bullet \to \bullet$ is: 

\begin{center}
\includegraphics[scale=0.25]{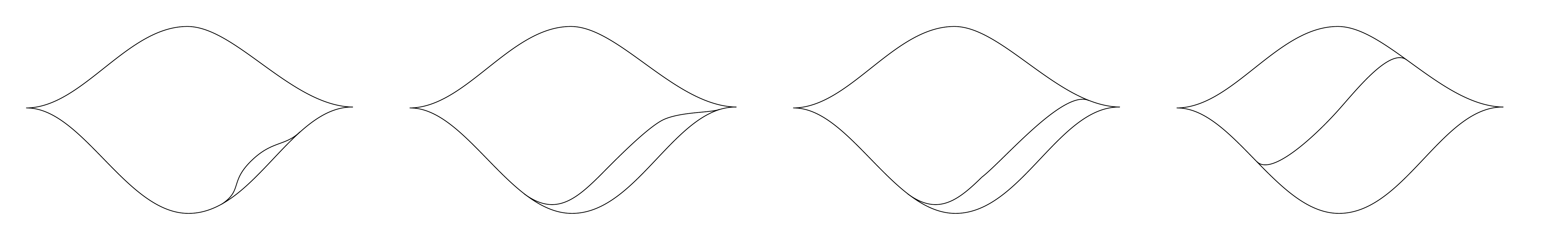}
\end{center}

This is just an isotopy.  (In terms of the Reidemeister moves above, we used move IV.) 

\vspace{4mm}

As usual, for shrubs we suspend the picture and take a union over appropriate permutations of coordinates. 
Away from the big unknot, this is still an isotopy: the hypersurfaces meet in the front projection, but where they meet
they are parallel to distinct coordinate hypersurfaces, hence have different lifts.  Near the big unknot, it is not an isotopy -- 
the loci where the hypersurfaces meet the big unknot will meet during the movie -- but the behavior along the big unknot 
is a cone of the nearby behavior, hence is a ribbotopy.

\begin{example}
Scenes from the movie for $\bullet \leftarrow \bullet \to \bullet$: 

\begin{center}
\includegraphics[scale=0.25]{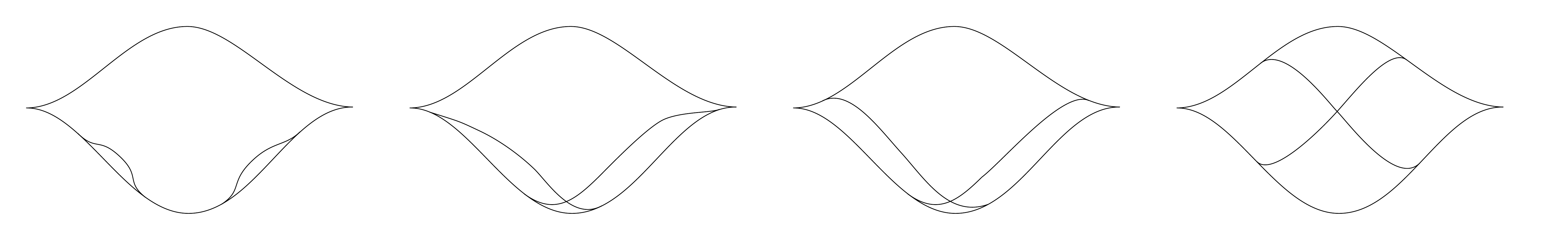}
\end{center}

All the action happens between the first scene and the second.  This is a ribbotopy of type `C' in the list above. 
\end{example}

We turn to the general case.  First, by induction, for the stuff within the pods in the mothership, ribbotope to the appropriate
arboreal singularity.  Then, with the upper boundaries of the pods, apply the above ribbotopy.  One must check that whatever 
is going on in the interior of the pods does not interfere with the moving hypersurfaces.  This is true because the stuff inside
the pods is already in the arboreal configuration --- i.e., the stuff in each pod is close to tangent to coordinate hypersurfaces,
and the coordinates involved for different pods are disjoint.  So away from the big unknot, the moving fronts of these 
hypersurfaces lift to disjoint legendrians.  This extends to the big unknot as a ribbotopy. 

\begin{example}
It's not so meaningful to draw the $\bullet \to \bullet \to \bullet$ case, since nothing could even conceivably interfere with
anything else, but here it is anyway.  

\begin{center}
\includegraphics[scale=0.25]{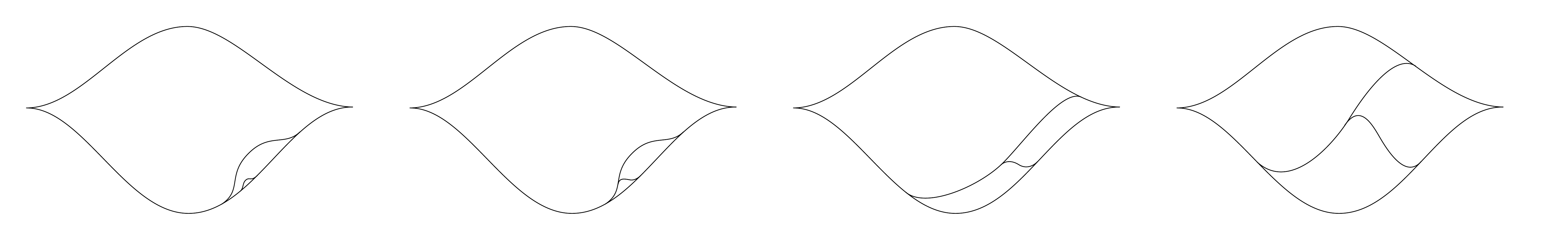}
\end{center}

Again, the movie is just an isotopy except between the  first scene and the second. 
\end{example}

\bibliographystyle{amsplain}
\bibliography{all}

\end{document}